\documentclass[11pt]{amsart}
\usepackage{amsmath,amssymb,amsthm}
\usepackage[hidelinks]{hyperref}
\newtheorem{thm}{Theorem}[section]

\newtheorem{cor}[thm]{Corollary}
\newtheorem{lem}[thm]{Lemma}

\theoremstyle{definition}

\numberwithin{equation}{section}

\def\Q{\mathbb{Q}}
\def\Z{\mathbb{Z}}
\def\al{\alpha}
\def\be{\beta}

\title[Perfect powers among Jacobsthal numbers]{Perfect powers among Jacobsthal numbers}

\author{Paulius Virbalas}
\address{Institute of Mathematics, Faculty of Mathematics and Informatics, Vilnius
University\\
Lithuania}
\email{paulius.virbalas@mif.vu.lt}
\begin{document}
\begin{abstract}
Jacobsthal numbers are an example of Lucas sequence defined by a recurrence relation analogous to that of the Fibonacci numbers but with different parameters. In this paper, we show that the only perfect powers among Jacobsthal numbers are the trivial ones, namely $0$ and $1$. Using Binet formulas, our problem reduces to an exponential Diophantine equation in three unknowns. It is resolved by the modular approach, following the framework developed by Bennett and Skinner for ternary Diophantine equations. 
\end{abstract}

\maketitle
\noindent\textbf{Note added in Version 2.}
After the first version of this manuscript was posted, it was brought to the author's attention that the main result follows as a special case of Theorem 3 of Laishram, Ngairangbam and Maibam \cite{lnm}. Moreover, the Diophantine equation upon which this manuscript is centered follows from Lemma~2.6 of \cite{lnm}, and also from Theorem~2 of Bugeaud and Mignotte \cite{bm}.

\section{Introduction}\label{Intro}

A non-negative integer $w$ is called a perfect power if there exist non-negative integers $x$ and $m\geq2$ such that $w=x^m$. The cases $w=0$ and $w=1$ are referred to as trivial perfect powers. Determining perfect powers in linear recurrence sequences is a classical problem in number theory. The most prominent example is the Fibonacci numbers, defined by the recurrence
\[F_0=0,~F_1=1,\quad F_n=F_{n-1}+F_{n-2}\]
for all positive integers $n\geq2$. In 2006, a landmark paper by Bugeaud, Mignotte, and Siksek \cite{bms} finally resolved a long-standing conjecture by proving that the only perfect powers among the Fibonacci numbers are
\[F_0=0,\quad F_1=F_2=1,\quad F_6=8, \quad F_{12}=144.\]
This breakthrough subsequently inspired extensive research on related Diophantine equations involving Fibonacci numbers, such as $F_n+1=x^m$ or $F_n-F_k=2^m$ (see, e.g., \cite{f9,f11,f2,f1,f7,f6,f8,f5,f4,f3,f10}).\par
The Fibonacci numbers are a particular case of the family of Lucas sequences. Given two integer parameters $P$ and $Q$, the Lucas sequence $U_n(P,Q)$ of the first kind is defined by
\[
U_0(P,Q)=0,\quad U_1(P,Q)=1,\quad 
U_n(P,Q)=P\,U_{n-1}(P,Q)-Q\,U_{n-2}(P,Q)
\]
for $n\geq2$. Note that the Fibonacci numbers correspond to $U_n(1,-1)$. Closely related to $U_n(P,Q)$ is its companion sequence $V_n(P,Q)$, called Lucas sequence of the second kind. It is defined by
\[
V_0(P,Q)=2,\quad V_1(P,Q)=P,\quad 
V_n(P,Q)=P\,V_{n-1}(P,Q)-Q\,V_{n-2}(P,Q)
\]
for $n\geq2$. The first terms of $V_n(1,-1)$ are 
\[2,1,3,4,7,11,18,29,47,76, ...\]
The terms of $V_n(1,-1)$ are known as Lucas numbers. In the same paper of \cite{bms}, it was shown that the only perfect powers in $V_n(1,-1)$ are $1$ and $4$. Other well-known examples of Lucas sequences are Pell numbers, corresponding to $U_n(2,-1)$, and the Pell-Lucas numbers, corresponding to the companion sequence $V_n(2,-1)$. In 1991, Peth\H{o} \cite{pel1} proved that the only perfect powers among Pell numbers are $0,1$ and $169$ (a substantially simpler proof was later found by Cohn \cite{pel2}), while in 2015, Bravo, Das, Guzm\'an, and Laishram \cite{pel3} showed that there are no perfect powers among Pell-Lucas numbers. Other results on perfect powers in linear recurrence sequences can be found in \cite{pw7,pw5, pw6,pw4,pw1,pw2,pw3,pw9,pw8}.\par
The characteristic polynomial of the sequences $U(P,Q)$ and $V(P,Q)$ is 
\[\lambda(x)=x^2-Px+Q\]
with the roots
\[\al=\frac{P+\sqrt{D}}{2} \quad\text{and}\quad \be=\frac{P-\sqrt{D}}{2},\]
where $D$ is the discriminant of $\lambda(x)$. The Binet formulas show closed-form expressions for the $n$-th term of Lucas sequences. In particular
\[ U_n(P,Q) = \frac{\alpha^n - \beta^n}{\alpha - \beta} 
\quad\text{and}\quad 
V_n(P,Q) = \alpha^n + \beta^n.
\]
Thus, the problem of determining perfect powers in Lucas sequences, via their Binet
formulas, reduces to the study of Diophantine equations of the form 
\[
x^m=U_n(P,Q) \quad \text{or} \quad x^m=V_n(P,Q),\quad x\geq0,~m\geq2.
\]
Depending on the choice of parameters $P$ and $Q$, the complexity of such equations 
varies considerably. For example, in the Fibonacci case, the problem reduces to the Diophantine equation
\[
x^m=\frac{1}{\sqrt{5}}\left(\left(\frac{1+\sqrt{5}}{2}\right)^n - \left(\frac{1-\sqrt{5}}{2}\right)^n\right).
\]
The proof that this equation admits integer solutions with $m \geq 2$ only for $n=0,1,2,6,12$ required an innovative combination of lower bounds for linear forms 
in three logarithms with the modular approach, extending techniques that were originally developed in the proof of Fermat's Last Theorem. In contrast, the $n$-th term of the sequence $U_n(3,2)$ is given by $M_n=2^n - 1$, which is the sequence of Mersenne numbers. The corresponding Diophantine equation
\[
x^m=2^n - 1
\]
has no solutions with $m \geq 2$. This follows directly from Mih\u{a}ilescu's proof 
of Catalan's conjecture \cite{mih02}, which states that the only non-trivial consecutive perfect powers are $3^2$ and $2^3$.\par
In this paper, we consider Lucas sequences with $P=1$ and $Q=-2$. The terms of $U_n(1,-2)$ are called Jacobsthal numbers (after German mathematician Ernst Jacobsthal) and are denoted by $J_n$, while the terms of $V_n(1,-2)$ are called Jacobsthal-Lucas numbers and are denoted by $j_n$. 
The Jacobsthal numbers satisfy the recurrence relation
\[J_0=0,~J_1=1,\quad J_n=J_{n-1}+2J_{n-2},\quad n\geq 2.\]
The sequence $J_n$ begins with
\begin{equation}\label{Jl1}
0,1,1,3,5,11,21,43,85,171,\ldots,
\end{equation}
and admits the Binet formula
\begin{equation}\label{nth}
J_n=\frac{2^n-(-1)^n}{3}.
\end{equation}
The Jacobsthal-Lucas numbers satisfy the recurrence relation
\[j_0=2,~j_1=1,\quad j_n=j_{n-1}+2j_{n-2},\quad n\geq 2.\]
The sequence $j_n$ begins with
\begin{equation}\label{jl1}
2,1,5,7,17,31,65,127,257,511,\ldots,
\end{equation}
and admits the Binet formula
\begin{equation}\label{jl2}
j_n=2^n+(-1)^n.
\end{equation}
It is known that Jacobsthal and Jacobsthal-Lucas numbers have applications in coding theory and graph theory \cite{aj3,aj1,aj2}. Numerical properties of these sequences can be found in \cite{pj2,pj1}, while various generalizations are discussed in \cite{gj3, gj2, gj1}. Several Diophantine problems involving Jacobsthal numbers have been studied in \cite{dj1,dj2}.\par
In the present work, we contribute to the research on perfect powers in linear recurrence sequences by investigating separately the Diophantine equations 
\[x^m = J_n\quad\text{and}\quad x^m = j_n.\]
In fact, due to Mih\u{a}ilescu's \cite{mih02} proof of Catalan's conjecture, we know that the equation $x^m=2^n + (-1)^n$ has no solutions in positive integers for $m\geq 2$. In view of \eqref{jl1} and \eqref{jl2}, this means that the only perfect power among Jacobsthal-Lucas numbers is $j_1=1$, corresponding to $1^m=2^1-1$. We prove that the analogous conclusion holds for the Jacobsthal numbers, namely that the only perfect powers are the trivial ones. In view of \eqref{nth}, this is a consequence of the following result.
\begin{thm}\label{thmain}
The Diophantine equation
\begin{equation*}
3x^m=2^n-(-1)^n
\end{equation*}
has no integer solution $(x,m,n) \in \Z^3$ with $m\geq2$ and $n\geq3$. 
\end{thm}
By inspecting the values of $J_n$ for $n=0,1,2$ in \eqref{Jl1}, we obtain the following corollary. 
\begin{cor}\label{cormain}
The only perfect powers among the Jacobsthal numbers are $0$ and $1$. 
\end{cor}
The main case of Theorem~\ref{thmain} is to prove that for $m\geq2$ the 
Diophantine equation
\[
3x^m=2^n+1
\]
has no integer solutions when $n\geq3$ is odd. By transforming this equation into
\begin{equation}\label{me}
3x^m-2^r\cdot (2^t)^m=1^2,
\end{equation}
we treat it in Lemma~\ref{wiles} by applying the results of Bennett and Skinner \cite{bs04} 
on ternary Diophantine equations of the form
\begin{equation}\label{put}
Ax^m+By^m=Cz^2.
\end{equation}
In their pioneering work, Bennett and Skinner showed that to any primitive 
solution $(x,y,z)=(a,b,c)$ of \eqref{put} with $m\geq7$ an odd prime, one can associate an elliptic 
curve $E$, called a Frey curve, whose discriminant is an $m$-th power up to a 
constant factor. Then, via Ribet's level-lowering techniques, they further showed that 
the corresponding Galois representation arises from a cuspidal newform of 
weight $2$ and some relatively small level $N$. For the equation \eqref{me} related to Jacobsthal numbers, we obtain $N=6$. Since it is 
known that there are no newforms of weight $2$ and level $6$, it follows that 
there is no non-trivial solution to \eqref{me}. The cases with small exponents, namely 
$m<7$, are handled by reducing them to Thue equations and solving them with a 
computer algebra system, similarly as in \cite[pp. 75]{bms}.\par
The paper is structured as follows. In Section~\ref{lem}, we prove two key lemmas, which allow for a more concise and transparent proof of Theorem~\ref{thmain} in Section~\ref{prf}.

\section{Lemmas}\label{lem}
If $n$ is odd, the $n$-th Jacobsthal number is given by
\[
J_n=\frac{2^n+1}{3}.
\]
Thus, if $x^m=J_n$, then the problem converts to the study of Diophantine equation
\[
3x^m-2^n=1.
\]
The following lemma treats the case $n=m$. 

\begin{lem}\label{smk}
Let $m\ge 5$ be a prime. Then the Diophantine equation
\begin{equation}\label{smkp}
3x^m-y^m=1
\end{equation}
has no integer solutions $(x,y)\in\Z^2$ with $xy\neq0$.
\end{lem}

\begin{proof}
Following Cohen \cite{cohen}, we refer to Diophantine equations of the form
\begin{equation}\label{smke}
x^m+L^r y^m+z^m=0,
\end{equation}
where $r<m$ and $m\ge5$ is prime, as Serre-Mazur-Kraus equations. According to \cite{cohen}, the solution $(x,y,z)=(a,b,c) \in \Z^3$ to \eqref{smke} is called non-trivial primitive if $abc\neq0$ and $a,b,c$ are pairwise coprime. From \cite[Theorem~15.5.3]{cohen} follows that equation \eqref{smke} has no non-trivial primitive solutions for $L=3$. Now suppose on the contrary, that $(x,y)=(a,b)$ is a solution to \eqref{smkp} with $ab\neq0$. Thus
\[3a^m-b^m=1,\]
which can be rewritten as
\[
(-b)^m+3^1\cdot a^m+(-1)^m=0.
\]
However, that implies that $(x,y,z)=(-b,a,-1)$ is a non-trivial primitive solution to \eqref{smke} with $L=3$ and $r=1$, a contradiction. Hence the proof is complete.
\end{proof}

The following lemma is necessary to tackle the case $n\ne m$. Its proof proceeds via the modular approach following the algorithm described by Bennett and Skinner \cite{bs04} to systematically investigate Diophantine equations of the form
\begin{equation}\label{ter}
Ax^m+By^m=Cz^2,
\end{equation}
where $m\ge 7$ is prime, and $A, B, C$ are fixed non-zero pairwise coprime integers. The ideas underlying the modular approach rely on deep results of Wiles \cite{wil}, Ribet \cite{rib}, Mazur \cite{maz}, and others (see, e.g., \cite{tay,dar,kub}). Accordingly, the introduction to the methods used requires a substantial amount of background. Since our argument is essentially a direct application of the algorithm, we do not reproduce the definitions or technical details here and instead refer the reader to \cite{bs04}.\par
According to \cite{bs04}, a solution $(x,y,z)=(a,b,c)\in \Z^3$ to \eqref{ter} is called primitive if $Aa$, $Bb$, and $Cc$ are all non-zero and pairwise coprime.
\begin{lem}\label{wiles}
Let $r$ be a fixed positive integer, and let $m\ge 7$ be prime. Then the Diophantine equation
\begin{equation}\label{mn}
3x^m-2^r y^m=z^2
\end{equation}
has no primitive integer solutions $(x,y,z)\in\mathbb{Z}^3$ such that $xy\neq1$ and $y$ even.
\end{lem}

\begin{proof}
Observe that \eqref{mn} matches the form of \eqref{ter} with
\[
A=3,\quad B=-2^r,\quad C=1.
\]
Suppose on the contrary, that there exist a primitive integer solution $(x,y,z)=(a,b,c)$ to \eqref{mn} such that $ab\neq1$ and $b$ is even . Thus
\begin{equation}\label{sol}
3a^m-2^r b^m=c^2.
\end{equation}
According to the classification given in \cite[Chapter~2]{bs04}, equation \eqref{sol} falls into class (v), characterized by the properties
\[\operatorname{ord}_2(Bb^m)\geq6\quad\text{and}\quad c\equiv C \pmod 4.\]
The condition $\operatorname{ord}_2(Bb^m)\geq6$ is satisfied due to the assumption that $b$ is even, and $m\geq7$. With respect to $c$, note that $c^2$ is necessarily odd (due to $3a^m$ and $2^rb^m$ being coprime) so that if $c$ is not equal to $1$ modulo $4$, then $c$ can be substituted for $-c$ to satisfy the condition $c\equiv C \pmod 4$. The algorithm proceeds further by associating to equations from class (v) a Frey elliptic curve
\begin{equation*}
E~:~Y^2+XY=X^3+\frac{cC-1}{4}X^2+\frac{BCb^m}{64}X
\end{equation*}
which in our case becomes
\begin{equation}\label{e3}
E~:~Y^2+XY=X^3+\frac{c-1}{4}X^2-\frac{2^r b^m}{64}X.
\end{equation}
To the elliptic curve $E$, according to \cite[Lemma 3.2-3.3]{bs04} we associate a number $N_m(E)$ (related to the conductor $N(E)$ of $E$), which in our case is obtained as follows
\begin{equation}\label{lvl}
N_m(E)=\prod_{p|C,~ p\neq m} p^2 \cdot \prod_{q|AB, ~q \neq m} q=3\cdot 2 =6,
\end{equation}
since in our case $C=1$ and $AB=-3\cdot 2^r$ with $r\geq1$.
Let $\rho_m^E$ be a Galois representation 
\[\rho_m^E:\operatorname{Gal}(\overline{\Q}/{\Q}) \to \operatorname{GL}_2(\mathbb{F}_m) \]
arising from the action on the $m$-torsion points $E[m]$ of $E$. Since $m\geq7$ is prime, and $ab\neq1$, \cite[Corollary~3.1 ]{bs04} implies that the representation $\rho_m^E$ is irreducible. Consequently, by \cite[Lemma 3.2 ]{bs04}, the representation $\rho_m^E$ arises from a cuspidal newform of weight $2$, level $N_m^E$, and trivial Nebentypus character. From \eqref{lvl} we have that $N_m^E=6$. However, this leads to a contradiction as there are no newforms of weight $2$ and level $6$ \cite[Proposition 4.1]{bs04}. Hence the proof is complete.
\end{proof}

\section{Proof of the Main Result}\label{prf}
\begin{proof}[Proof of Theorem~\ref{thmain}]
Suppose on the contrary, that Diophantine equation
\begin{equation}\label{de}
3x^m=2^n-(-1)^n
\end{equation}
has at least one integer solution $(x,m,n) \in \Z^3$ with $m\geq2$ and $n\geq3$. \par 
We begin with the case when $n$ is even. Thus, \eqref{de} becomes $3x^m=2^n-1$. Since $n\geq3$, we have that $n=2k$ for some positive integer $k>1$. Consequently  
\[3x^m=(2^{k}-1)(2^{k}+1).\]
It is clear that $x$ must be odd. Hence $2^{k}-1$ and $2^{k}+1$ are coprime. Therefore
\[2^{k}-1=3^sx_1^m\quad\text{and}\quad 2^{k}+1=3^{1-s} x_2^m,\]
where $s\in \{ 0,1\}$. Consequently, either $2^k-1=x_1^m$ or $2^k+1=x_2^m$. From Mih\u{a}ilescu's proof of Catalan's conjecture \cite{mih02} it follows that the only possible solution to either of these equations is $2^3+1=3^2$. However, if $k=3$, then $x^m=21$, which is impossible due to $m\geq2$. This settles the case when $n$ is even.\par
We continue with the case when $n$ is odd. Thus, \eqref{de} becomes 
\begin{equation}\label{refq}
3x^m=2^n+1.
\end{equation}
Since $2^n+1 \equiv 1 \pmod{4}$ it follows that $3x^m\equiv 1 \pmod{4}$. This, and the fact that $x$ is clearly odd, imply that $m$ is odd too. If $m$ is not prime, then $x^m=(x^l)^p$ for some prime odd divisor of $m$, which means that from any solution $(x,m,n)$ one can obtain a solution $(x^l,p,n)$. Therefore, it is sufficient to prove that \eqref{refq} has no solutions of the required type, when $m$ is an odd prime. 
Also, due to $3x^m=2^n+1\geq 2^2+1=5$, and the fact that $x$ is odd, we can assume that $x>2$. Thus 
\[2^n+1=3x^m > 3 \cdot 2^m.\]
Therefore, $n>m$. Write $n=tm+r$, where $t\geq1$ is some positive integer and $r \in \{0, \ldots, m-1\}$. Then \eqref{refq} can be rewritten as
\begin{equation}\label{rew}
3x^m-2^r\cdot (2^t)^{m}=1.
\end{equation}
First we deal with the case, when $m\geq7$ (also recall that $m$ can be assumed to be prime). If \eqref{rew} 
holds for $r=0$ with some $x=a$, then $(x,y)=(a,2^t)$ is a solution to 
\begin{equation}\label{sup1}
3x^m-y^m=1
\end{equation}
with $xy\neq0$, which was shown to be impossible in Lemma~\ref{smk}. If \eqref{rew} 
holds with $r \in \{1,\ldots, m-1\}$ with some $x=a$, then $(x,y,z)=(a,2^t,1) \in \Z^3$ is a solution 
\begin{equation}\label{sup2}
3x^m-2^r y^m=z^2
\end{equation}
with $xy\neq1$ and $y$ even. Moreover, this solution is primitive, as
as $3a, 2^r\cdot 2^t, 1$ and are clearly non-zero and pairwise coprime (recall that $a$ is odd and $t\geq1$). However, this is again a contradiction, as in Lemma~\ref{wiles} it was demonstrated, that such solutions do not exist.\par  
To conclude the proof, it remains to consider equation \eqref{rew} when $m<7$ is an odd prime, i.e, $m \in \{3,5\}$. In order to treat these cases it suffices to solve eight Thue equations with any suitable computer algebra system. Indeed, if $m=5$, then we have five equations to consider, namely
\[3x^5-2^r\cdot y^5=1,~~\text{where}~~r=0,1,2,3,4. \]
If $m=3$, then we have three equations
\[3x^3-2^r\cdot y^3=1,~~\text{where}~~r=0,1,2. \]
By solving these equations in the computer algebra system PARI/GP (for example, $3x^5-4y^5=1$ is solved using the command \verb|thue(3*x^5-4,1)|) we find that 
the only possible solutions $(x,y)\in\Z^2$ that satisfy at least one of these equations are $(1,1)$, $(-1,-1)$ or $(0,-1)$. However, since $y=2^t$ with $t\geq1$, we always have $2^t\neq \pm 1$. This completes the proof of the theorem.
\end{proof}

\begin{proof}[Proof of Corollary~\ref{cormain}]
The $n$-th Jacobsthal number is given by the formula
\[J_n=\frac{2^n-(-1)^n}{3},\quad n=0,1,2\ldots\]
If $x^m=J_n$ for some $m\geq2$, then 
\[3x^m=2^n-(-1)^n.\]
By Theorem~\ref{thmain}, this equation has no integer solutions with $m\geq2$ and $n\geq3$. Thus, no $J_n$ is a perfect power when $n\geq3$. On the other hand
\[J_0=0,~J_1=1,~J_2=1.\]
Therefore, $0$ and $1$ are the only perfect powers among the Jacobsthal numbers.
\end{proof}
\bibliographystyle{acm}
\bibliography{references}

@article {lnm,
    AUTHOR = {Laishram, Shanta and Ngairangbam, Sudhir Singh and Maibam,
              Ranjit Singh},
     TITLE = {Yet another generalization of {S}ylvester's theorem and its
              application},
   JOURNAL = {Publ. Math. Debrecen},
  FJOURNAL = {Publicationes Mathematicae Debrecen},
    VOLUME = {95},
      YEAR = {2019},
    NUMBER = {1-2},
     PAGES = {1--17},
      ISSN = {0033-3883,2064-2849},
   MRCLASS = {11N13 (11B25 11B39 11D61)},
  MRNUMBER = {3998023},
MRREVIEWER = {Florian\ Luca},
       DOI = {10.5486/pmd.2019.8217},
       URL = {https://doi.org/10.5486/pmd.2019.8217},
}

@incollection {bm,
    AUTHOR = {Bugeaud, Yann and Mignotte, Maurice},
     TITLE = {On the {D}iophantine equation {$(x^n-1)/(x-1)=y^q$} with
              negative {$x$}},
 BOOKTITLE = {Number theory for the millennium, {I} ({U}rbana, {IL}, 2000)},
     PAGES = {145--151},
 PUBLISHER = {A K Peters, Natick, MA},
      YEAR = {2002},
      ISBN = {1-56881-126-8},
   MRCLASS = {11D61},
  MRNUMBER = {1956223},
MRREVIEWER = {Natarajan\ Saradha},
}

@book {cohen,
    AUTHOR = {Cohen, Henri},
     TITLE = {Number theory. {V}ol. {II}. {A}nalytic and modern tools},
    SERIES = {Graduate Texts in Mathematics},
    VOLUME = {240},
 PUBLISHER = {Springer, New York},
      YEAR = {2007},
     PAGES = {xxiv+596},
      ISBN = {978-0-387-49893-5},
   MRCLASS = {11-01 (11D61 11F80 11J86 11Mxx)},
MRREVIEWER = {R.\ C.\ Baker},
}

@article {bs04,
    AUTHOR = {Bennett, Michael A. and Skinner, Chris M.},
     TITLE = {Ternary {D}iophantine equations via {G}alois representations
              and modular forms},
   JOURNAL = {Canad. J. Math.},
  FJOURNAL = {Canadian Journal of Mathematics. Journal Canadien de
              Math\'ematiques},
    VOLUME = {56},
      YEAR = {2004},
    NUMBER = {1},
     PAGES = {23--54},
      ISSN = {0008-414X,1496-4279},
   MRCLASS = {11D41 (11F11 11F80)},
MRREVIEWER = {Henri\ Darmon},
       DOI = {10.4153/CJM-2004-002-2},
       URL = {https://doi.org/10.4153/CJM-2004-002-2},
}

@article {mih02,
    AUTHOR = {Mih\u{a}ilescu, Preda},
     TITLE = {Primary cyclotomic units and a proof of {C}atalan's
              conjecture},
   JOURNAL = {J. Reine Angew. Math.},
  FJOURNAL = {Journal f\"ur die Reine und Angewandte Mathematik. [Crelle's
              Journal]},
    VOLUME = {572},
      YEAR = {2004},
     PAGES = {167--195},
      ISSN = {0075-4102,1435-5345},
   MRCLASS = {11D61 (11R18 11R27)},
MRREVIEWER = {Ren\'e\ Schoof},
       DOI = {10.1515/crll.2004.048},
       URL = {https://doi.org/10.1515/crll.2004.048},
}

@article {bms,
    AUTHOR = {Bugeaud, Yann and Mignotte, Maurice and Siksek, Samir},
     TITLE = {Classical and modular approaches to exponential {D}iophantine
              equations. {I}. {F}ibonacci and {L}ucas perfect powers},
   JOURNAL = {Ann. of Math. (2)},
  FJOURNAL = {Annals of Mathematics. Second Series},
    VOLUME = {163},
      YEAR = {2006},
    NUMBER = {3},
     PAGES = {969--1018},
      ISSN = {0003-486X,1939-8980},
   MRCLASS = {11D61 (11B39 11D59 11J86)},
MRREVIEWER = {Yuri\ Bilu},
       DOI = {10.4007/annals.2006.163.969},
       URL = {https://doi.org/10.4007/annals.2006.163.969},
}

@incollection {pel1,
    AUTHOR = {Peth\H{o}, A.},
     TITLE = {The {P}ell sequence contains only trivial perfect powers},
 BOOKTITLE = {Sets, graphs and numbers ({B}udapest, 1991)},
    SERIES = {Colloq. Math. Soc. J\'anos Bolyai},
    VOLUME = {60},
     PAGES = {561--568},
 PUBLISHER = {North-Holland, Amsterdam},
      YEAR = {1992},
      ISBN = {0-444-98681-2},
   MRCLASS = {11D61 (11B37)},
MRREVIEWER = {T.\ N.\ Shorey},
}

@article {pel2,
    AUTHOR = {Cohn, J. H. E.},
     TITLE = {Perfect {P}ell powers},
   JOURNAL = {Glasgow Math. J.},
  FJOURNAL = {Glasgow Mathematical Journal},
    VOLUME = {38},
      YEAR = {1996},
    NUMBER = {1},
     PAGES = {19--20},
      ISSN = {0017-0895,1469-509X},
   MRCLASS = {11D61 (11B37)},
  
MRREVIEWER = {Nikos\ Tzanakis},
       DOI = {10.1017/S0017089500031207},
       URL = {https://doi.org/10.1017/S0017089500031207},
}

@article {pel3,
    AUTHOR = {Bravo, Jhon J. and Das, Pranabesh and Guzm\'an, Sergio and
              Laishram, Shanta},
     TITLE = {Powers in products of terms of {P}ell's and {P}ell-{L}ucas
              sequences},
   JOURNAL = {Int. J. Number Theory},
  FJOURNAL = {International Journal of Number Theory},
    VOLUME = {11},
      YEAR = {2015},
    NUMBER = {4},
     PAGES = {1259--1274},
      ISSN = {1793-0421,1793-7310},
   MRCLASS = {11B37 (11D61 15A15)},
  
MRREVIEWER = {Refik\ Keskin},
       DOI = {10.1142/S1793042115500682},
       URL = {https://doi.org/10.1142/S1793042115500682},
}

@article {dj1,
    AUTHOR = {Gaber, Ahmed},
     TITLE = {Intersections of {P}ell, {P}ell-{L}ucas numbers and sums of
              two {J}acobsthal numbers},
   JOURNAL = {Punjab Univ. J. Math. (Lahore)},
  FJOURNAL = {The Punjab University. Journal of Mathematics},
    VOLUME = {55},
      YEAR = {2023},
    NUMBER = {5-6},
     PAGES = {241--252},
      ISSN = {1016-2526},
   MRCLASS = {11B39 (11J86)},

       DOI = {10.1112/blms.12724},
       URL = {https://doi.org/10.1112/blms.12724},
}

@article {dj2,
    AUTHOR = {Gaber, Ahmed and Ahmed, Mohiedeen},
     TITLE = {Solutions of the {D}iophantine equations {$B_r = J_s + J_t$}
              and {$C_r = J_s + J_t$}},
   JOURNAL = {J. Math.},
  FJOURNAL = {Journal of Mathematics},
      YEAR = {2023},
     PAGES = {Art. ID 8851478, 8},
      ISSN = {2314-4629,2314-4785},
   MRCLASS = {11D61 (11B39 11J86)},
 
       DOI = {10.1155/2023/8851478},
       URL = {https://doi.org/10.1155/2023/8851478},
}

@article {pj1,
    AUTHOR = {Nelsen, Roger B.},
     TITLE = {An introduction to the {J}acobsthal numbers},
   JOURNAL = {College Math. J.},
  FJOURNAL = {The College Mathematics Journal},
    VOLUME = {55},
      YEAR = {2024},
    NUMBER = {2},
     PAGES = {153--158},
      ISSN = {0746-8342,1931-1346},
   MRCLASS = {11B39},
  
       DOI = {10.1080/07468342.2023.2282201},
       URL = {https://doi.org/10.1080/07468342.2023.2282201},
}

@article {pj2,
    AUTHOR = {\c{C}elik, Song\"ul and Durukan, \.Inan and \"Ozkan, Engin},
     TITLE = {New recurrences on {P}ell numbers, {P}ell-{L}ucas numbers,
              {J}acobsthal numbers, and {J}acobsthal-{L}ucas numbers},
   JOURNAL = {Chaos Solitons Fractals},
  FJOURNAL = {Chaos, Solitons \& Fractals},
    VOLUME = {150},
      YEAR = {2021},
     PAGES = {Paper No. 111173, 8},
      ISSN = {0960-0779,1873-2887},
   MRCLASS = {11B39},
  
       DOI = {10.1016/j.chaos.2021.111173},
       URL = {https://doi.org/10.1016/j.chaos.2021.111173},
}

@article {aj1,
    AUTHOR = {Kulo\u{g}lu, Bahar and \"Ozkan, Engin},
     TITLE = {Applications of {J}acobsthal and {J}acobsthal-{L}ucas numbers
              in coding theory},
   JOURNAL = {Math. Montisnigri},
  FJOURNAL = {Mathematica Montisnigri. Matematika Tsrne Gore},
    VOLUME = {57},
      YEAR = {2023},
     PAGES = {54--64},
      ISSN = {0354-2238,2704-4963},
   MRCLASS = {94B35 (11B39)},
  
       DOI = {10.20948/mathmontis-2023-57-4},
       URL = {https://doi.org/10.20948/mathmontis-2023-57-4},
}

@article {aj2,
    AUTHOR = {\"Otele\c{s}, Ahmet and Karatas, Zekeriya Y. and Zangana, Diyar
              O. Mustafa},
     TITLE = {Jacobsthal numbers and associated {H}essenberg matrices},
   JOURNAL = {J. Integer Seq.},
  FJOURNAL = {Journal of Integer Sequences},
    VOLUME = {21},
      YEAR = {2018},
    NUMBER = {2},
     PAGES = {Art. 18.2.5, 10},
      ISSN = {1530-7638},
   MRCLASS = {05A15 (11B83 15B99)},
 
MRREVIEWER = {Radica\ Boji\v ci\'c},
}

@article {aj3,
    AUTHOR = {Bruhn, Henning and Gellert, Laura and G\"unther, Jacob},
     TITLE = {Jacobsthal numbers in generalized {P}etersen graphs},
   JOURNAL = {J. Graph Theory},
  FJOURNAL = {Journal of Graph Theory},
    VOLUME = {84},
      YEAR = {2017},
    NUMBER = {2},
     PAGES = {146--157},
      ISSN = {0364-9024,1097-0118},
   MRCLASS = {05C70 (05C15 05C30)},
  
MRREVIEWER = {Ioan\ Tomescu},
       DOI = {10.1002/jgt.22017},
       URL = {https://doi.org/10.1002/jgt.22017},
}

@article {gj1,
    AUTHOR = {Br\'od, Dorota and Michalski, Adrian},
     TITLE = {On generalized {J}acobsthal and {J}acobsthal-{L}ucas numbers},
   JOURNAL = {Ann. Math. Sil.},
  FJOURNAL = {Annales Mathematicae Silesianae},
    VOLUME = {36},
      YEAR = {2022},
    NUMBER = {2},
     PAGES = {115--128},
      ISSN = {0860-2107,2391-4238},
   MRCLASS = {11B39},
  
       DOI = {10.2478/amsil-2022-0011},
       URL = {https://doi.org/10.2478/amsil-2022-0011},
}

@article {gj2,
    AUTHOR = {Babada\u{g}, Faik},
     TITLE = {A new approach to {J}acobsthal, {J}acobsthal-{L}ucas numbers
              and dual vectors},
   JOURNAL = {AIMS Math.},
  FJOURNAL = {AIMS Mathematics},
    VOLUME = {8},
      YEAR = {2023},
    NUMBER = {8},
     PAGES = {18596--18606},
      ISSN = {2473-6988},
   MRCLASS = {11B39},
  
       DOI = {10.3934/math.2023946},
       URL = {https://doi.org/10.3934/math.2023946},
}

@article {gj3,
    AUTHOR = {Al-Kateeb, Ala'a},
     TITLE = {A generalization of {J}acobsthal and {J}acobsthal-{L}ucas
              numbers},
   JOURNAL = {Jordan J. Math. Stat.},
  FJOURNAL = {Jordan Journal of Mathematics and Statistics},
    VOLUME = {14},
      YEAR = {2021},
    NUMBER = {3},
     PAGES = {467--481},
      ISSN = {2075-7905,2227-5487},
   MRCLASS = {11B39},
 
}

@article {wil,
    AUTHOR = {Wiles, Andrew},
     TITLE = {Modular elliptic curves and {F}ermat's last theorem},
   JOURNAL = {Ann. of Math. (2)},
  FJOURNAL = {Annals of Mathematics. Second Series},
    VOLUME = {141},
      YEAR = {1995},
    NUMBER = {3},
     PAGES = {443--551},
      ISSN = {0003-486X,1939-8980},
   MRCLASS = {11G05 (11D41 11F11 11F80 11G18)},
  
MRREVIEWER = {Karl\ Rubin},
       DOI = {10.2307/2118559},
       URL = {https://doi.org/10.2307/2118559},
}

@article {rib,
    AUTHOR = {Ribet, K. A.},
     TITLE = {On modular representations of {${\rm Gal}(\overline{\bf
              Q}/{\bf Q})$} arising from modular forms},
   JOURNAL = {Invent. Math.},
  FJOURNAL = {Inventiones Mathematicae},
    VOLUME = {100},
      YEAR = {1990},
    NUMBER = {2},
     PAGES = {431--476},
      ISSN = {0020-9910,1432-1297},
   MRCLASS = {11G18 (11F32 11F80 11S37)},
  
MRREVIEWER = {Glenn\ Stevens},
       DOI = {10.1007/BF01231195},
       URL = {https://doi.org/10.1007/BF01231195},
}

@article {maz,
    AUTHOR = {Mazur, B.},
     TITLE = {Rational isogenies of prime degree (with an appendix by {D}.
              {G}oldfeld)},
   JOURNAL = {Invent. Math.},
  FJOURNAL = {Inventiones Mathematicae},
    VOLUME = {44},
      YEAR = {1978},
    NUMBER = {2},
     PAGES = {129--162},
      ISSN = {0020-9910,1432-1297},
   MRCLASS = {14K07 (10D35 14G25)},
 
MRREVIEWER = {V.\ V.\ Shokurov},
       DOI = {10.1007/BF01390348},
       URL = {https://doi.org/10.1007/BF01390348},
}

@article {tay,
    AUTHOR = {Breuil, Christophe and Conrad, Brian and Diamond, Fred and
              Taylor, Richard},
     TITLE = {On the modularity of elliptic curves over {$\bold Q$}: wild
              3-adic exercises},
   JOURNAL = {J. Amer. Math. Soc.},
  FJOURNAL = {Journal of the American Mathematical Society},
    VOLUME = {14},
      YEAR = {2001},
    NUMBER = {4},
     PAGES = {843--939},
      ISSN = {0894-0347,1088-6834},
   MRCLASS = {11G05 (11F80 11G07 14G35)},
  
MRREVIEWER = {Karl\ Rubin},
       DOI = {10.1090/S0894-0347-01-00370-8},
       URL = {https://doi.org/10.1090/S0894-0347-01-00370-8},
}

@article {kub,
    AUTHOR = {Kubert, Daniel Sion},
     TITLE = {Universal bounds on the torsion of elliptic curves},
   JOURNAL = {Proc. London Math. Soc. (3)},
  FJOURNAL = {Proceedings of the London Mathematical Society. Third Series},
    VOLUME = {33},
      YEAR = {1976},
    NUMBER = {2},
     PAGES = {193--237},
      ISSN = {0024-6115,1460-244X},
   MRCLASS = {10B15 (14G25)},
  
MRREVIEWER = {J.\ W. S. Cassels},
       DOI = {10.1112/plms/s3-33.2.193},
       URL = {https://doi.org/10.1112/plms/s3-33.2.193},
}

@article {dar,
    AUTHOR = {Darmon, Henri and Merel, Lo\"ic},
     TITLE = {Winding quotients and some variants of {F}ermat's last
              theorem},
   JOURNAL = {J. Reine Angew. Math.},
  FJOURNAL = {Journal f\"ur die Reine und Angewandte Mathematik. [Crelle's
              Journal]},
    VOLUME = {490},
      YEAR = {1997},
     PAGES = {81--100},
      ISSN = {0075-4102,1435-5345},
   MRCLASS = {11G18 (11D41 11F80 11G05)},
  
MRREVIEWER = {Kenneth\ Kramer},
       DOI = {10.1515/crll.1997.490.81},
       URL = {https://doi.org/10.1515/crll.1997.490.81},
}

@article {f1,
    AUTHOR = {Irmak, Nurettin and Szalay, L\'aszl\'o},
     TITLE = {On the equation {$F_n-F_m=F_t^a$}},
   JOURNAL = {Bol. Soc. Mat. Mex. (3)},
  FJOURNAL = {Bolet\'in de la Sociedad Matem\'atica Mexicana. Third Series},
    VOLUME = {31},
      YEAR = {2025},
    NUMBER = {3},
     PAGES = {Paper No. 105, 12},
      ISSN = {1405-213X,2296-4495},
   MRCLASS = {11J86 (11B39 11D61)},
  
       DOI = {10.1007/s40590-025-00776-y},
       URL = {https://doi.org/10.1007/s40590-025-00776-y},
}

@article {f2,
    AUTHOR = {Hasanalizade, Elchin},
     TITLE = {Sums of {F}ibonacci numbers close to a power of 2},
   JOURNAL = {Bol. Soc. Mat. Mex. (3)},
  FJOURNAL = {Bolet\'in de la Sociedad Matem\'atica Mexicana. Third Series},
    VOLUME = {29},
      YEAR = {2023},
    NUMBER = {1},
     PAGES = {Paper No. 19, 11},
      ISSN = {1405-213X,2296-4495},
   MRCLASS = {11B39},
       DOI = {10.1007/s40590-023-00490-7},
       URL = {https://doi.org/10.1007/s40590-023-00490-7},
}

@article {f3,
    AUTHOR = {\c{S}iar, Zafer},
     TITLE = {On the exponential {D}iophantine equation {$F^x_n \pm F^x_m =
              a$} with {$a\in \{F_r, L_r\}$}},
   JOURNAL = {Int. J. Number Theory},
  FJOURNAL = {International Journal of Number Theory},
    VOLUME = {19},
      YEAR = {2023},
    NUMBER = {1},
     PAGES = {41--57},
      ISSN = {1793-0421,1793-7310},
   MRCLASS = {11B39 (11D61 11J86)},
MRREVIEWER = {Gopal\ Krishna\ Panda},
       DOI = {10.1142/S1793042123500021},
       URL = {https://doi.org/10.1142/S1793042123500021},
}

@article {f4,
    AUTHOR = {Tripathy, Bibhu Prasad and Patel, Bijan Kumar},
     TITLE = {Sums of three {F}ibonacci numbers close to a power of 2},
   JOURNAL = {Integers},
  FJOURNAL = {Integers. Electronic Journal of Combinatorial Number Theory},
    VOLUME = {23},
      YEAR = {2023},
     PAGES = {Paper No. A4, 13},
      ISSN = {1553-1732},
   MRCLASS = {11B39 (11J86)},

MRREVIEWER = {Sai\ Gopal\ Rayaguru},
}

@article {f5,
    AUTHOR = {Rihane, Salah Eddine},
     TITLE = {On the {D}iophantine equations {$F_N = P_M \pm P_\ell$} and
              {$P_N = F_M \pm F_\ell$} involving {F}ibonacci and {P}ell
              numbers},
   JOURNAL = {Integers},
  FJOURNAL = {Integers. Electronic Journal of Combinatorial Number Theory},
    VOLUME = {22},
      YEAR = {2022},
     PAGES = {Paper No. A100, 12},
      ISSN = {1553-1732},
   MRCLASS = {11B39 (11J86)},
 
MRREVIEWER = {Florian\ Luca},
}

@article {f6,
    AUTHOR = {Luca, Florian and Tchammou, Euloge and Togb\'e, Alain},
     TITLE = {On the exponential {D}iophantine equation
              {$F_n^x+F_{n+1}^x+\cdots+F_{n+k-1}^x=F_m$}},
   JOURNAL = {Funct. Approx. Comment. Math.},
  FJOURNAL = {Uniwersytet im. Adama Mickiewicza w Poznaniu. Wydzia\l\
              Matematyki i Informatyki. Functiones et Approximatio
              Commentarii Mathematici},
    VOLUME = {66},
      YEAR = {2022},
    NUMBER = {2},
     PAGES = {139--159},
      ISSN = {0208-6573,2080-9433},
   MRCLASS = {11B39 (11J86)},
 
MRREVIEWER = {Pagdame\ Tiebekabe},
       DOI = {10.7169/facm/1860},
       URL = {https://doi.org/10.7169/facm/1860},
}

@article {f7,
    AUTHOR = {Kihel, Omar and Larone, Jesse},
     TITLE = {On the nonnegative integer solutions of the equation {$F_n\pm
              F_m=y^a$}},
   JOURNAL = {Quaest. Math.},
  FJOURNAL = {Quaestiones Mathematicae. Journal of the South African
              Mathematical Society},
    VOLUME = {44},
      YEAR = {2021},
    NUMBER = {8},
     PAGES = {1133--1139},
      ISSN = {1607-3606,1727-933X},
   MRCLASS = {11B39 (11D61 11J86)},
 
MRREVIEWER = {Mahadi\ Ddamulira},
       DOI = {10.2989/16073606.2020.1775155},
       URL = {https://doi.org/10.2989/16073606.2020.1775155},
}

@article {f8,
    AUTHOR = {Patel, Bijan Kumar and Chaves, Ana Paula},
     TITLE = {On the exponential {D}iophantine equation
              {$F_{n+1}^x-F_{n-1}^x=F_m$}},
   JOURNAL = {Mediterr. J. Math.},
  FJOURNAL = {Mediterranean Journal of Mathematics},
    VOLUME = {18},
      YEAR = {2021},
    NUMBER = {5},
     PAGES = {Paper No. 187, 11},
      ISSN = {1660-5446,1660-5454},
   MRCLASS = {11B39 (11D61 11J86)},
 
MRREVIEWER = {Nurettin\ Irmak},
       DOI = {10.1007/s00009-021-01831-4},
       URL = {https://doi.org/10.1007/s00009-021-01831-4},
}

@article {f9,
    AUTHOR = {B\'erczes, Attila and Hajdu, Lajos and Luca, Florian and Pink,
              Istv\'an},
     TITLE = {On the {D}iophantine equation {$F_n{}^x +F_k{}^ x =F_m{}^y$}},
   JOURNAL = {Ramanujan J.},
  FJOURNAL = {Ramanujan Journal. An International Journal Devoted to the
              Areas of Mathematics Influenced by Ramanujan},
    VOLUME = {67},
      YEAR = {2025},
    NUMBER = {2},
     PAGES = {Paper No. 29, 37},
      ISSN = {1382-4090,1572-9303},
   MRCLASS = {11B39},
  
MRREVIEWER = {Devbhadra\ V.\ Shah},
       DOI = {10.1007/s11139-025-01075-w},
       URL = {https://doi.org/10.1007/s11139-025-01075-w},
}

@article {f10,
    AUTHOR = {Zhang, Zhongfeng and Togb\'e, Alain},
     TITLE = {On the {D}iophantine equation
              {$F_{n-2}^m+F_n^m+F_{n+2}^m=y^p$}},
   JOURNAL = {Funct. Approx. Comment. Math.},
  FJOURNAL = {Uniwersytet im. Adama Mickiewicza w Poznaniu. Wydzia\l\
              Matematyki i Informatyki. Functiones et Approximatio
              Commentarii Mathematici},
    VOLUME = {71},
      YEAR = {2024},
    NUMBER = {2},
     PAGES = {229--235},
      ISSN = {0208-6573,2080-9433},
   MRCLASS = {11D61 (11B39)},
 
MRREVIEWER = {L\'aszl\'o\ Szalay},
       DOI = {10.7169/facm/2130},
       URL = {https://doi.org/10.7169/facm/2130},
}

\end{document}